\documentclass[11pt]{article}
\usepackage{amssymb,amsmath,amsfonts,amsthm,epsfig,latexsym,color}
\usepackage{amscd, mathrsfs}
\usepackage{latexsym, euscript, epic, eepic}
\usepackage{graphicx}
\usepackage{verbatim}
 \usepackage[all]{xy}

\makeatletter
\newcommand{\subsectionruninhead}{\@startsection{subsection}{2}{0mm}
{-\baselineskip}{-0mm}{\bf\large}}
\newcommand{\subsubsectionruninhead}{\@startsection{subsubsection}{3}{0mm}
{-\baselineskip}{-0mm}{\bf\normalsize}}
\makeatother

\newtheorem*{theorem*}{Theorem}
\newtheorem{theoremalph}{Theorem}

\newtheorem*{proposition*}{Proposition}
\newtheorem*{corollary*}{Corollary}
\newtheorem*{conjecture*}{Conjecture}
\newtheorem*{claim*}{Claim}
\newtheorem*{remark*}{Remark}
\newtheorem{theorem}{Theorem}[section]

\newtheorem{proposition}[theorem]{Proposition}

\newtheorem{lemma}[theorem]{Lemma}

\theoremstyle{definition}
\newtheorem{definition}[theorem]{Definition}
\newtheorem{remark}[theorem]{Remark}

\numberwithin{equation}{section}

 \def\RR{{\mathbb R}}

\newcommand{\supp}{\operatorname{supp}}

\newcommand{\orb}{\operatorname{Orb}}

\newcommand{\ind}{\operatorname{Ind}}
\newcommand{\sing}{\operatorname{Sing}}

\newcommand{\topo}{\operatorname{top}}

\begin{document}

\title{Measures of intermediate entropies for star vector fields}

\author{Ming Li\footnote{\scriptsize M. Li is supported by NSFC 11571188 and the Fundamental Research Funds for the Central Universities.}, 
	    Yi Shi\footnote{\scriptsize Y. Shi is supported by NSFC 11701015 and Young Elite Scientists Sponsorship Program by CAST.}, 
	    Shirou Wang\footnote{\scriptsize S. Wang is supported by NSFC 11771026 and 11471344, PIMS PTCS and a PIMS CRG
	    	grant.}
	    and Xiaodong Wang\footnote{\scriptsize X. Wang is supported by NSFC 11701366 and Shanghai Sailing Program 17YF1409300.}
    }

\maketitle

\begin{abstract}
We prove that all star vector fields, including Lorenz attractors and multisingular hyperbolic vector fields, admit the intermediate entropy property. To be precise, if $X$ is a star vector field with $h_{\topo}(X)>0$, then for any $h\in [0,h_{\topo}(X))$, there exists an ergodic invariant measure $\mu$ of $X$ such that $h_{\mu}(X)=h$. Moreover, we show that the topological entropy is lower semi-continuous for star vector fields.
\end{abstract}

\section{Introduction}

The concept of entropy was introduced by Kolmogorov in 1958, which has been the most important invariant in ergodic theory and dynamical systems during the past 60 years. It reflects the complexity of the dynamical system. In some circumstances, the positivity of entropy forces the system to have typical structure. For instance, Katok proved the following milestone theorem.

\begin{theorem*}[\cite{Katok2}]\label{Thm:Katok}
	Let $f$ be a $C^r(r>1)$ surface diffeomorphism, and $\mu$ be an ergodic measure of $f$ with $h_{\mu}(f)>0$. Then for any $\varepsilon>0$, there exits a hyperbolic horseshoe $\Lambda_\varepsilon$ satisfying $h_{\topo}(f,\Lambda_\varepsilon)>h_\mu(f)-\varepsilon$.
\end{theorem*}

Similar result holds for higher dimensional diffeomorphisms having a hyperbolic measure $\mu$ with positive measure entropy, see \cite{KM}.
Combined with the variational principle, this theorem implies the lower semi-continuity of the topological entropy for $C^r(r>1)$ surface diffeomorphisms. Moreover, a hyperbolic horseshoe is conjugated to a full shift. Thus every $C^r(r>1)$ surface diffeomorphism $f$ has the {\it intermediate entropy property}, i.e. for any constant $h\in[0,h_{\topo}(f))$, there exists an ergodic measure $\mu$ of $f$ satisfying $h_{\mu}(f)=h$.

Katok raised the following conjecture.
\begin{conjecture*}\label{Que:Katok}
	Every $C^r(r\geq1)$ diffeomorphism $f$ on a manifold satisfies the intermediate entropy property.
\end{conjecture*}

Not all systems admit the intermediate entropy property. There are uniquely ergodic homeomorphisms \cite{HK} with positive topological entropy, see also \cite{bcr}. It seems that the system is required to have some smoothness to achieve this property. For diffeomorphisms, this conjecture is widely open. Herman \cite{H} constructed a $C^{\infty}$ minimal diffeomorphism with positive topological entropy. However, Herman's example is not uniquely ergodic and has intermediate entropy property, so it is not a counterexample to Katok's conjecture. For certain skew products, Sun~\cite{sun-skew} verifies this conjecture to be true. One can refer to~\cite{sun-skew2,sun-toral,guan-sun-wu} for related results.

In this paper, we verify Katok's conjecture for star vector fields. Let $M$ be a $d$-dimensional closed Riemannian manifold.
Denote by $\mathscr{X}^1(M)$ the space of $C^1$ vector fields on $M$ endowed with the $C^1$ topology. For any $X\in\mathscr{X}^1(M)$, denote by $\phi_t^X$ (or $\phi_t$ if there are no confusions) the $C^1$ flow on $M$ generated by $X$.

During the long march to the stability conjecture, Liao~\cite{liao-shadowing1} and Ma\~n\'e~\cite{mane} noticed an important class of systems, which was named by Liao the {\it star systems}.
We call $X\in\mathscr{X}^1(M)$ a {\it star vector field}, if there is a $C^1$ neighborhood $\mathcal{U}$ of $X$ such that for any $Y\in\mathcal{U}$, all singularities and all periodic orbits of $\phi_t^Y$ are hyperbolic. The set of all star vector fields on $M$ is denoted by $\mathscr{X}^*(M)$ which is endowed with $C^1$-topology. Notice that $\mathscr{X}^*(M)$ is an open set in $\mathscr{X}^1(M)$. We can define star diffeomorphisms similarly.

The notion of star diffeomorphisms plays a key role in proving the famous stability conjecture. Actually, a diffeomorphism is star if and only if it is hyperbolic, i.e. satisfies Axiom A plus no cycle condition~\cite{aoki,hayashi}. Gan-Wen \cite{gan-wen} proved that nonsingular star vector fields satisfy Axiom A and the no-cycle condition.
However, a singular star vector field may fail to satisfy Axiom A, for instance, the famous Lorenz attractor \cite{G}.

To describe the geometric structure of Lorenz attractor, Morales, Pacifico
and Pujals \cite{MPP} developed a notion called singular hyperbolicity. See \cite{lgw} and \cite{MM} for higher dimensions. Then \cite{sgw} showed that a generic star vector field is singular hyperbolic under some homogeneous assumption for singularities.

For a long time, people believed that singular hyperbolicity is an appropriate notion to describe generic star vector fields, just like hyperbolicity to star diffeomorphisms.
Recently, however, da Luz \cite{daLuz} constructed an exciting example of 5-dimensional star vector field, which has two singularities with different indices robustly contained in one chain class, thus is robustly non-singular hyperbolic. Due to this example, Bonatti and da Luz \cite{bonatti-da luz} developed a new notion called multisingular hyperbolicity. They showed that every $X$ in an open dense subset of $\mathscr{X}^*(M)$ is multisingular hyperbolic. Conversely, every multisingular hyperbolic vector field is star.

Our main result is the following theorem.

\begin{theoremalph}\label{Thm:maintheorem}
	Every star vector field satisfies the intermediate entropy property.
\end{theoremalph}

Every multisingular hyperbolic vector field has the intermediate entropy property since it is star \cite{bonatti-da luz}. Recall that the Lorenz attractor is star and expansive \cite{APPV}, so it has a measure with maximal entropy. This implies the following corollary.

\begin{corollary*}
	Let $X$ be a vector field which defines the Lorenz attractor $\Lambda$. Then for every $h\in[0,h_{\topo}(X,\Lambda)]$, there exists an ergodic measure $\mu$ supported on $\Lambda$ satisfying $h_{\mu}(X)=h$.
\end{corollary*}

When we study the dynamics of vector fields, even for star vector fields, the main difficulty is that the flow speed tends to zero when the orbit is close to singularities. This obstructs us to get uniform estimations for some structure of vector fields. For instance, we could not have uniform sizes of stable and unstable manifolds for star vector fields with singularities.

In order to prove Theorem~\ref{Thm:maintheorem}, an important step is to find hyperbolic sets of the star vector field away from singularities. Moreover, the topological entropies restricted on these hyperbolic sets can be arbitrarily close to the topological entropy of the vector field. This will help us to conquer the trouble caused by singularities. Since a hyperbolic set is persistent under $C^1$-perturbations, we have the following theorem, which states that the topological entropy of star vector fields is lower semi-continuous.

\begin{theoremalph}\label{Thm:LSC of entropy for star flows}
	The topological entropy function $h_{\topo}(\cdot):\mathscr{X}^*(M)\rightarrow\RR$ is lower semi-continuous.
\end{theoremalph}

Theorem \ref{Thm:LSC of entropy for star flows} implies that the topological entropy function is lower semi-continuous for Lorenz attractors and multisingular hyperbolic vector fields. The lower semi-continuity of toplogical entropy for systems rely on certain hyperbolicity, for instance, the surface diffeomorphisms and nonuniformly hyperbolic systems \cite{Katok2,KM}. Our result also follows the idea of Katok's argument \cite{Katok2}. The main novelty of our proof is using Liao's techniques \cite{liao-shadowing1,liao-shadowing2} to deal with the difficulty caused by singularities.

\subsection{Proof of main theorems}

In this subsection, we prove Theorem \ref{Thm:maintheorem} and Theorem \ref{Thm:LSC of entropy for star flows} based on two propositions.
We first find a hyperbolic set, whose entropy could approximate the whole topologically entropy of vector field. This hyperbolic set is a suspension of a horseshoe, which is topological conjugated to a suspension of a full shift. Then we consider the intermediate entropy property for the suspension of a full shift.

Recall that given a full shift with $k$-symbols $(\Sigma_k,\sigma)$ and a continuous function $\varphi:\Sigma_k\rightarrow \mathbb{R}^+$, one can define the {\it $\varphi$-suspension space}:
$$\Sigma_k^\varphi=\{(x,t):x\in \Sigma_k, t\in [0,\varphi(x)]\}/(x,\varphi(x))\sim(\sigma(x),0).$$
The {\it suspension flow} over $(\Sigma_k,\sigma)$ is defined as $\sigma^\varphi_t:\Sigma_k^\varphi\rightarrow \Sigma_k^\varphi$ by $\sigma^\varphi_t(x,s)=(x,s+t)$.
The function $\varphi$ is called the {\it roof function} of the suspension flow.

Let $\phi_t:M\rightarrow M$ be a $C^1$ flow generated by $X\in\mathscr{X}^1(M)$. A compact $\phi_t$-invariant set $\Lambda$ is called a {\it horseshoe} of $\phi_t$, if there exists a suspension flow $\sigma^\varphi_t:\Sigma_k^\varphi\rightarrow \Sigma_k^\varphi$ with continuous roof function $\varphi$, and a homeomorphism $\pi:\Sigma_k^{\varphi}\rightarrow\Lambda$, such that
$$
\phi_t\circ\pi=\pi\circ\sigma_t^{\varphi}.
$$
We can see that a horseshoe of $\phi_t$ does not contain any singularities. Actually, it has positive distance to $\sing(X)$, which is the set of all singularities of $X$.

Theorem \ref{Thm:LSC of entropy for star flows} is a direct corollary of the following proposition, which states that the topological entropy of a star vector field can be approximated by the topological entropy of hyperbolic horseshoes. The proof of this proposition will be given in Section \ref{Section:lower-semi-continuity}.

\begin{proposition}\label{Prop:lower-semi-continuity of entropy}
	Assume $\phi_t$ is the $C^1$ flow generated by $X\in\mathscr{X}^*(M)$ and $h_{\topo}(X)>0$. For every $\varepsilon>0$, there is a hyperbolic horseshoe $\Lambda_{\varepsilon}$ of $\phi_t$, such that
	$$
	h_{\topo}(X,\Lambda_{\varepsilon})>h_{\topo}(X)-\varepsilon.
	$$
\end{proposition}

\begin{remark*}
	The hyperbolic horseshoe $\Lambda_{\varepsilon}$ is away from $\sing(X)$.
	This proposition implies that when we remove a very small neighborhood of all singularities, the entropy of a star vector field restricted on the maximal invariant set of the rest part of the manifold can approximate the whole topological entropy. This seems natural, because the vanishing of flow speed when orbits are close to singularities does not generate complexity.
	
	For instance, let $X$ be a vector field which defines the Lorenz attractor $\Lambda$ with attracting region $U$. For the singularity $\sigma\in\Lambda$, and every $\delta>0$, we denote by
	$$
	\Lambda_{\delta}=\bigcap_{t\in\RR}\phi_t(U\setminus B_{\delta}(\sigma)),
	$$
	where $B_{\delta}(\sigma)$ is the $\delta$-neighborhood of $\sigma$. Here $\Lambda_{\delta}$ is a hyperbolic set and we have
	$$
	h_{\topo}(X, \Lambda_{\delta}) \longrightarrow h_{\topo}(X,\Lambda) \qquad {\rm as}\qquad \delta\rightarrow0.
	$$
\end{remark*}

It is well known that topological conjugacy preserves topological entropy. So to prove Theorem \ref{Thm:maintheorem}, we only need to show that the suspension flow of a full shift has the intermediate entropy property. Actually, we prove that the suspension flow of shift of finite type(SFT) has intermediate entropy property, see Proposition \ref{Pro:intermediate entropy of suspension flow}.

\begin{proposition}\label{Pro:intermediate entropy of full shift flow}
	Let $(\Sigma_k^\varphi,\sigma_t^\varphi)$ be a suspension flow of a full shift with $k\geq2$. Then for any constant $h\in(0,h_{\topo}(\sigma_t^\varphi,\Sigma_k^\varphi))$, there exists a $\sigma_t^\varphi$-ergodic measure $\tilde{\mu}$ satisfying $$h_{\tilde{\mu}}(\sigma_t^\varphi,\Sigma_k^\varphi)=h.$$
\end{proposition}

\begin{remark*}
	The intermediate entropy property of full shift is obvious. Even the measure of suspension has a 1-1 correspondence with the measure of full shift. However, since the roof function is not a constant in general, entropies of the two corresponding measures are not linearly dependent. See the beginning of Section \ref{Section:proof of proposition 1.5}.
\end{remark*}

Now we prove Theorem \ref{Thm:maintheorem} by admitting Proposition \ref{Prop:lower-semi-continuity of entropy} and Proposition \ref{Pro:intermediate entropy of full shift flow}, whose proofs will be given in Section \ref{Section:lower-semi-continuity} and Section \ref{Section:suspension flow of SFT} respectively.

\begin{proof}[ Proof of Theorem~\ref{Thm:maintheorem}]
	Let $\phi_t: M\rightarrow M$ be the flow generated by $X\in\mathscr{X}^*$ and assume that $h_{\topo}(X)>0$.
	Fix any constant $h\in [0,h_{\topo}(X))$. We aim to prove that there is an ergodic measure $\mu$ of $X$ such that $h_{\mu}(X)=h$.

	If $h=0$, notice that if $X$ has no singularities, then $\phi_t$ must admit periodic orbits by $\cite{gan-wen}$. Take $\mu$ to be the Dirac measure supported on a periodic orbit or a singularity of $X$, then $h_{\mu}(X)=0$.
	
	If $h\in (0,h_{\topo}(X))$, by Proposition~\ref{Prop:lower-semi-continuity of entropy}, there is a hyperbolic horseshoe $\Lambda$ such that $h_{\topo}(X,\Lambda)>h$.
	Assume the suspension of full shift $(\Sigma_k^\varphi,\sigma_t^\varphi)$ conjugates to $(\Lambda,\phi_t)$, then
	$$
	h_{\topo}(\sigma_t^\varphi,\Sigma_k^\varphi)=h_{\topo}(X,\Lambda).
	$$
	By Proposition~\ref{Pro:intermediate entropy of full shift flow}, there exists an ergodic measure $\tilde{\mu}$ of $(\Sigma_k^\varphi,\sigma_t^\varphi)$ satisfying $h_{\tilde{\mu}}(\sigma_t^\varphi,\Sigma_k^\varphi)=h$. Then we conclude Theorem~\ref{Thm:maintheorem} through the conjugation.
\end{proof}

\subsection*{Acknowledgements}
The authors would like to thank Sylvain Crovisier, Shaobo Gan, Peng Sun, Xiao Wen, Dawei Yang and Jiagang Yang for many useful discussions and suggestions.~

\bigskip

\section{Lower semi-continuity of topological entropy for star flows}\label{Section:lower-semi-continuity}

In this section, we prove Proposition~\ref{Prop:lower-semi-continuity of entropy}. Firstly, we introduce some known results of star vector fields.

\subsection{Preliminaries of star vector fields}


Given $X\in\mathscr{X}^1(M)$. Recall that we denote by $\phi_t:=\phi^X_t:M\rightarrow M$ the $C^1$-flow generated by $X$ and $\Phi_t= \textmd{d}\phi_t:TM\rightarrow TM$ the tangent flow of $\phi_t$.

For any point $x\in M\setminus \sing(X)$, denote by $\mathcal{N}_x$ the orthogonal complement space of the flow direction $X(x),$ i.e.
$$\mathcal{N}_x=\{v\in T_xM: v\bot X(x)\}.$$
Denote by $\mathcal{N}=\bigcup_{x\in M\setminus \sing(X)}\mathcal{N}_x$. The {\it linear Poincar\'e flow} $\psi_t:\mathcal{N}\rightarrow\mathcal{N}$ of $X$ is defined as: $$\psi_t(v)=\Phi_t(v)-\frac{\langle\Phi_t(v),X(\phi_t(x))\rangle}{|X(\phi_t(x))|^2}X(\phi_t(x)),$$
where $x\in M\setminus\sing(X)$, $v\in\mathcal{N}_x$ and $\psi_t(v)\in\mathcal{N}_{\phi_t(x)}$.

Fix $T>0$. It is easy to see that the norm
$$
\|\psi_T\|=\sup\{|\psi_T(v)|: v\in\mathcal{N}, |v|=1\}
$$
is uniformly upper bounded on $\mathcal N$, although $M\setminus \sing(X)$ may be not compact. Denote by
$$
m(\psi_T)=\inf\{|\psi_T(v)|: v\in\mathcal{N}, |v|=1\}
$$
the mininorm of $\psi_T$. Since $m(\psi_T)=\|\psi_T^{-1}\|^{-1}=\|\psi_{-T}\|^{-1}$, the mininorm $m(\psi_T)$ is uniformly bounded from $0$ on $\mathcal N$ for any fixed $T>0$.

We also need the so called {\it scaled linear Poincar\'e flow} $\psi^*_t:\mathcal{N}\rightarrow\mathcal{N}$, which is defined as:
$$\psi^*_t(v)=\frac{|X(x)|}{|X(\phi_t(x))|}\psi_t(v)=\frac{1}{\|\Phi_t|_{\langle X(x)\rangle}\|}\psi_t(v),$$
where $x\in M\setminus\sing(X)$, $v\in\mathcal{N}_x$ and $\langle X(x)\rangle$ is the 1-dimensional subspace of $T_xM$ generated by the flow direction $X(x)$.

For every $x\in M\setminus\sing(X)$ and any $\delta>0$ small, we can define the normal manifold of $x$ to be
$$
N_x(\delta)=\exp_x(\mathcal{N}_x(\delta)),
$$
where $\mathcal{N}_x(\delta)=\{v\in\mathcal{N}_x:|v|\leq\delta\}$. It is clear that when $\delta$ is small enough, $N_x(\delta)$ is an imbedded submanifold which is diffeomorphic to $\mathcal{N}_x(\delta).$
Moreover, $N_x(\delta)$ is a local cross section transverse to the flow.

For any $T>0$ and $x\in M\setminus\sing(X),$ the flow $\phi_t$ defines a local holonomy map which is called the \emph{Poincar\'e map}:
$$
{\cal P}_{x,\phi_T(x)}: N_x(\delta)\longrightarrow N_{\phi_T(x)}(\delta'),
$$
where $\delta$ and $\delta'$ depend on the choice of $x$ and $T.$ Moreover, it is not hard to see that
$$
D_x{\cal P}_{x,\phi_T(x)}=\psi_T|_{\mathcal{N}_x}:\mathcal{N}_x\longrightarrow\mathcal{N}_{\phi_T(x)}.
$$

The following lemma is a direct corollary of Lemma 2.3 in \cite{gan-yang}.

\begin{lemma}\label{Lem:section-size}
	Given $X\in\mathscr{X}^1(M)$ and $T>0$, there exists $\delta_T>0$ such that for any $x\in M\setminus\sing(X)$, the Poincar\'e map
	$$
	{\cal P}_{x,\phi_T(x)}:N_x(\delta_T|X(x)|)\longrightarrow N_{\phi_T(x)}(\delta_TC_T|X(\phi_T(x))|),
	$$
	is well defined, where $C_T=\|\psi_T\|$ is a bounded constant.
\end{lemma}

Moreover, we have the following lemma which states the uniform continuity of Poincar\'e map up to flow speed.

\begin{lemma}[Lemma 2.4 in \cite{gan-yang}]\label{Lem:Continuity}
	Given $X\in\mathscr{X}^1(M)$ and $T>0$. Shrinking $\delta_T>0$ in Lemma \ref{Lem:section-size} if necessary, for any $x\in M\setminus\sing(X)$, consider the Poincar\'e map
	$$
	{\cal P}_{x,\phi_T(x)}:N_x(\delta_T|X(x)|)\longrightarrow N_{\phi_T(x)}(\delta_TC_T|X(\phi_T(x))|),
	$$
	we have that $D{\cal P}_{x,\phi_T(x)}$ is uniformly continuous, i.e. for any $\epsilon>0,$ there exists $\rho\in(0,\delta_T]$ such that for any $x\in M\setminus\sing(X)$ and $y,z\in N_x(\delta_T|X(x)|)$ with $|y-z|\leq\rho|X(x)|,$ we have
	$$
	|D_y{\cal P}_{x,\phi_T(x)}-D_z{\cal P}_{x,\phi_T(x)}|<\epsilon.
	$$
\end{lemma}

For the Poincar\'e  map, we have the following estimation of return time.
\begin{lemma}[Lemma 4.5 in \cite{wyz}]\label{Lem:Time-control}
	Given $X\in\mathscr{X}^1(M),$ let $\delta=\delta_1$ and $C_1$ be the constants in Lemma \ref{Lem:section-size} associated to the time-one map $\phi_1.$ Then by shrinking $\delta$ if necessary, there exits $\theta>0$ such that for any $x\in M\setminus\sing(X)$ and $y\in N_x(\delta|X(x)|),$ there exists a unique $t=t(y)\in(0,2)$ satisfying the following
	$$	\phi_t(y)\in N_{\phi_1(x)}(\delta C_1|X(\phi_1(x))|) \qquad and \qquad |t(y)-1|<\theta\cdot d(x,y).
	$$
\end{lemma}

\begin{definition}\label{Def:dominated-splitting}
	Given a vector field $X\in\mathscr{X}^1(M)$. Let $\Lambda$ be a $\phi_t$-invariant set with $\Lambda\cap\sing(X)=\emptyset$. We say $\Lambda$ admits a {\it dominated splitting} with respect to $\psi_t$,  if there is a $\psi_t$-invariant splitting $\mathcal{N}_{\Lambda}=E\oplus F$ and two constants $\eta>0$ and $T>0$, such that:
	$$\|\psi_T|_{E_x}\|\cdot \|\psi_{-T}|_{F_{\phi_T(x)}}\|\leq  e^{-\eta T}, \qquad\forall x\in\Lambda.$$
	To be precise, we also call this splitting an {\it $(\eta,T)$-dominated splitting}.
	
	Similarly, we say that a $\psi^*_t$-invariant splitting $\mathcal{N}_{\Lambda}=E\oplus F$ is  \emph{$(\eta,T)$-dominated} with respect to $\psi^*_t$ if $\psi_T$  in  the above inequality is replaced by $\psi_{T}^*$.
\end{definition}

\begin{remark}
	In general, $\Lambda$ is only $\phi_t$-invariant, but not compact. Moreover, recall that
	$$\psi^*_t(v)=\frac{|X(x)|}{|X(\phi_t(x))|}\psi_t(v)=\frac{1}{\|\Phi_t|_{\langle X(x)\rangle}\|}\psi_t(v),
	\qquad \forall x\in\Lambda, ~\forall v\in\mathcal{N}_x,$$
	we have that $\mathcal{N}_{\Lambda}=E\oplus F$ is an $(\eta,T)$-dominated splitting with respect to $\psi_t$ if and only if it is an $(\eta,T)$-dominated splitting with respect to $\psi^*_t$.
\end{remark}

Let $\mu$ be an ergodic measure of $X$. We say $\mu$ is {\em nontrivial} if it is not supported on a singularity or a periodic orbit.
A trivial ergodic measure $\mu$ is {\it hyperbolic} if its support is hyperbolic.
A nontrivial ergodic measure $\mu$ of $X$ is {\it hyperbolic} if it has a unique vanishing Lyapunov exponent which is associated to the flow direction $\langle X \rangle.$ The number of negative Lyapunov exponents of a hyperbolic ergodic measure $\mu$ is called the {\it index} of $\mu$ and denoted by $\ind(\mu)$.

It is proved in~\cite[Theorem 5.6]{sgw} that every ergodic measure of a star vector field is hyperbolic. The following lemma asserts the existence of dominated splitting on $\mathcal{N}_{\supp(\mu)\setminus\sing(X)}$ for any ergodic measure $\mu$ of a star vector field with respect to the scaled linear Poincar\'e flow.

\begin{lemma}\label{Lem:domination of SLPF}
	Given $X\in\mathscr{X}^*(M)$. Assume that $\mu$ is a non-trivial ergodic measure of $X$. Then the scaled linear Poincar\'e flow $\psi_t^*$ admits a dominated splitting $\mathcal{N}_{\supp(\mu)\setminus\sing(X)}=E\oplus F$ satisfying that $\dim(E)=\ind(\mu)$. Moreover, there exist two constants $\eta,T>0$, such that
	$$
	\int\log\|\psi^*_T|_{E_x}\|{\rm d}\mu(x)<-\eta, \qquad and \qquad
	\int\log\|\psi^*_{-T}|_{F_x}\|{\rm d}\mu(x)<-\eta.
	$$
\end{lemma}

\begin{proof}
	By applying the ergodic closing lemma \cite{mane}, there exists a sequence of vector fields $X_n\rightarrow X$, and periodic orbits $\gamma_n$ of $X_n$ with same index of $\mu$, such that $\gamma_n\rightarrow\supp(\mu)$ in the Hausdorff topology. Recall that a star vector field is far from homoclinic tangency, then from \cite[Corollary 2.11]{gan-yang}, the scaled linear Poincar\'e flow $\psi_t^*$ admits a dominated splitting $\mathcal{N}_{\supp(\mu)\setminus\sing(X)}=E\oplus F$ satisfying that $\dim(E)=\ind(\mu)$.
	
	Finally, the estimations of integrals of $\log\|\psi^*_T|_{E_x}\|$ and $\log\|\psi^*_{-T}|_{F_x}\|$ on $\mu$ have been showed in the proof of \cite[Theorem 5.6]{sgw}.
\end{proof}



\begin{definition}\label{Def:*-contracting}
	Given $X\in\mathscr{X}^1(M)$ and constants $C,\eta,T>0$. Let $\Lambda$ be a $\phi_t$-invariant set and $E\subset\mathcal{N}_{\Lambda\setminus\sing(X)}$ be an invariant subbundle of the linear Poincar\'e flow $\psi_t$. A point $x\in\Lambda\setminus\sing(X)$ is called $(C,\eta,T,E)$-$\psi^*_t$-contracting if there exists a partition:
	$$
	0=t_0<t_1<\cdots<t_n<\cdots, \qquad {\rm with}~ t_{n+1}-t_n\leq T, ~\forall n\in\mathbb{N},
	$$
	and $t_n\rightarrow\infty$ as $n\rightarrow\infty$, such that for any $n\in\mathbb{N}$,
	$$
	\prod_{i=0}^{n-1}\|\psi^*_{t_{i+1}-t_i}|_{E_{\phi_{t_i}(x)}}\|\leq C\cdot e^{-\eta t_n}.
	$$
	A point $x\in\Lambda\setminus\sing(X)$ is called $(C,\eta,T,E)$-$\psi^*_t$-expanding if it is $(C,\eta,T,E)$-$\psi^*_t$-contracting for $-X$.
\end{definition}

Due to the loss flow speed, even when the Poincar\'e flow is uniformly contracting, we could not have uniform size of stable manifolds if the orbits are approaching to singularities. However, the new theorem was proved in the pioneering work of Liao \cite{liao-shadowing2}, which allowed us to get the uniform size of stable manifolds up to flow speed. We recommend \cite[Section 2.5]{gan-yang} for a detailed and beautiful explanation of this theorem.

\begin{theorem}\label{Thm:Stable-Mfd}
	Let $X\in\mathscr{X}^1(M)$ and $\Lambda$ be a compact invariant set. Given constants $C,\eta,T>0$, assume that $\mathcal{N}_{\Lambda\setminus\sing(X)}=E\oplus F$ is an $(\eta,T)$-dominated splitting with respect to the linear Poincar\'e flow $\psi_t$. Then for any $\epsilon>0$, there exists $\delta>0$, such that if $x$ is $(C,\eta,T,E)$-$\psi^*_t$-contracting, then there exists a $C^1$-map $\kappa:E_x(\delta|X(x)|)\rightarrow\mathcal{N}_x$ satisfying that
	\begin{itemize}
		\item $d_{C^1}(\kappa,id)<\epsilon$;
		\item $\kappa(0)=0$, and  ${\rm Image}(\kappa)$ is tangent to $E_{x}$ at $0$;
		\item $W^{cs}_{\delta|X(x)|}(x)\subset W^s(\orb(x))$, where $W^{cs}_{\delta|X(x)|}(x)=\exp_x({\rm Image}(\kappa))$.
	\end{itemize}
	Moreover, if $0=t_0<t_1<\cdots<t_n<\cdots$ is the partition of $x$ for $(C,\eta,T,E)$-$\psi^*_t$-contraction, then there exists a constant $C'>0$, such that for any $n\in\mathbb{N}$, denote the constant $\delta_n=C'e^{-\frac{\eta n}{2}}\delta$, 
	there exists a $C^1$-map
	$$
	\kappa_n:E_{\phi_{t_n}(x)}(\delta_n|X(\phi_{t_n}(x))|)\rightarrow\mathcal{N}_{\phi_{t_n}(x)},
	$$
	which satisfies
	\begin{itemize}
		\item $d_{C^1}(\kappa_n,id)<\epsilon$;
		\item $\kappa_n(0)=0$, and ${\rm Image}(\kappa_n)$ is tangent to $E_{\phi_{t_n}(x)}$ at $0$;
		\item denoted $W^{cs}_{\delta_n|X(x)|}(\phi_{t_n}(x))=\exp_{\phi_{t_n}(x)}({\rm Image}(\kappa_n))$, one has
		$$
		{\cal P}_{x,\phi_{t_n}(x)}(W^{cs}_{\delta|X(x)|}(x))\subset W^{cs}_{\delta_n|X(\phi_{t_n}(x))|}(\phi_{t_n}(x)).
		$$
	\end{itemize}
\end{theorem}

\begin{remark*}
	The first part of Theorem~\ref{Thm:Stable-Mfd} is proved in \cite{liao-shadowing2}. The second part has been showed in \cite[Section 2]{gan-yang} with the exponenetial contracting of stable manifolds (\cite{HPS,Wen}) up to flow speed.
\end{remark*}

\bigskip

\subsection{Horseshoes with large entropy}
In this section, we prove that for any nontrivial ergodic measure of a star flow we could obtain a horseshoe whose entropy approximates its metric entropy. The main idea originates from \cite{Katok2} for $C^r(r>1)$ surface diffeomorphisms, but we need to conquer the difficulty caused by singularities.
A similar result can be found in~\cite{yang} with the assumption that the tangent flow $\Phi_t$ admits a dominated splitting.

\begin{proposition}\label{prop:LSC of entropy}
	Let $\phi_t$ be a star flow generated by $X\in\mathscr{X}^*(M)$ and $\mu$ be an ergodic measure satisfying that $h_{\mu}(X)>0$.
	Then for any $\varepsilon>0$, there is a hyperbolic horseshoe $\Lambda_{\varepsilon}$ of $\phi_t$ such that $h_{\topo}(X,\Lambda_{\varepsilon})>h_{\mu}(X)-\varepsilon$.
\end{proposition}

\begin{remark*}
    The idea of proof is the same with Katok's argument in \cite{Katok2} and \cite{KM}. The main difficulty arises from vanishing of flow speed when the orbit approaches singularities, and we lose the influence from hyperbolicity of typical points of the hyperbolic measure. However, due to the key observation of Liao, the hyperbolicity still holds in a flow speed size neighborhood of a typical point.
\end{remark*}

\begin{proof}
	Except a countable many $T\in\mathbb{R}$, the measure $\mu$ is ergodic with respect to the time-$T$ map of the flow \cite{pugh-shub}. Without loss of generality, we assume that $\mu$ is an ergodic measure for the time-1 map $\phi_1$ and the general case is identical. Note that $\mu$ is not supported on singularities since $h_{\mu}(X)>0$.
	
	By Lemma \ref{Lem:domination of SLPF} (we assume $T=1$ for simplicity and the general case is identical), there exists a constant $\eta>0$ such that
	$$
	\int\log\|\psi^*_1|_{E_x}\|{\rm d}\mu(x)<-\eta, \qquad {\rm and}\qquad
	\int\log\|\psi^*_{-1}|_{F_x}\|{\rm d}\mu(x)<-\eta.
	$$
	By Birkhoff ergodic theorem, for $\mu$-almost every $x\in M$ we have
	$$
	\lim_{k\rightarrow\infty}\frac{1}{k}\sum_{i=0}^{k-1}\log\|\psi^*_1|_{E_{\phi_i(x)}}\|=\int\log\|\psi^*_1|_{E_x}\|{\rm d}\mu(x)<-\eta,
	$$
	$$
	\lim_{k\rightarrow\infty}\frac{1}{k}\sum_{i=0}^{k-1}\log\|\psi^*_{-1}|_{F_{\phi_{-i}(x)}}\|=\int\log\|\psi^*_{-1}|_{F_x}\|{\rm d}\mu(x)<-\eta.
	$$
	
	Just as in the Pesin theory, for every $C>0$, let $\Lambda_C\subset\supp(\mu)$ be the set of points in $\supp(\mu)$ that are both $(C,\eta,1,E)$-$\psi^*_t$-contracting and $(C,\eta,1,F)$-$\psi^*_t$-expanding. That is for any $x\in\Lambda_C$ and any $n>0$, one has	
	$$
	\prod_{i=0}^{k-1}\|\psi^*_1|_{E_{\phi_i(x)}}\|\leq Ce^{-k\eta} \qquad{\rm and}\qquad
	\prod_{i=0}^{k-1}\|\psi^*_{-1}|_{F_{\phi_{-i}(x)}}\|\leq Ce^{-k\eta}.
	$$
	Note that $\mu(\Lambda_C)\rightarrow 1$ as $C\rightarrow+\infty.$ We choose $C$ large enough such that $\mu(\Lambda_C)>0.$
	Since $\mu(\sing(X))=0$, there exists $\delta_0>0$ such that  $\Delta_C:=\Lambda_C\setminus B_{\delta_0}(\sing(X))$ has positive $\mu$-measure, where $B_{\delta_0}(\sing(X))$ denotes the $\delta_0$-neighborhood of $\sing(X).$
	
	
	
	Now we fix a point $z\in\supp(\mu|_{\Delta_C})$, and take $\delta$ satisfying
	$$
	0<\delta\ll\min\{\frac{\delta_1|X(z)|}{1000},~\frac{\delta_0|X(z)|}{1000},
	~\frac{\varepsilon}{2\theta(h_{\mu}(\phi_t)-\varepsilon)}\},
	$$
	where $\delta_1$ is the constant in Lemma \ref{Lem:section-size}, \ref{Lem:Continuity}, \ref{Lem:Time-control} for $T=1$, and $\theta$ is the constant in Lemma \ref{Lem:Time-control}.
	
	We consider the local cross section
	$$
	N_z(\delta)=\exp_z(E_z(\delta)\times F_z(\delta)) \qquad {\rm and}\qquad V_z(\delta)=\phi_{(-\delta,\delta)}(N_z(\delta))=\bigcup_{t\in(-\delta,\delta)}\phi_t(N_z(\delta)).
	$$
    Similarly we define $N_z(\delta/2)$ and $V_z(\delta/2)$.
	Then we have $\mu(V_z(\delta)\cap\Delta_C)>0$.
	
	Following the classical argument of Katok \cite[Theorem 4.3]{Katok2}, 
	for $l,\beta>0$ and $n\in\mathbb{N}$, there exists a finite set $K_n(\beta,l)$ (the $(n,\frac{1}{l})$-separated set) satisfying the following:

	\begin{itemize}
		\item[--] $K_n(\beta,l)\subset V_z(\delta/2)\cap\Delta_C$;
		\item[--] $d_n(x,y)=\max_{i=0}^{n-1}\{d(\phi_i(x),\phi_i(y))\}>1/l \text{ for any } x, y\in K_n(\beta,l) \text{ with }x\neq y$;
		\item[--] for every $x\in K_n(\beta,l)$, there exists an integer $m_x$ with $n\leq m_x< (1+\beta)n$ such that
		$$
		\phi_{m_x}(x)\in V_z(\delta/2)\cap\Delta_C, \qquad {\rm and} \qquad d(x,\phi_{m_x}(x))<\frac{1}{1000l};
		$$
		\item[--] the following estimation satisfies
		$$
		\lim_{l\rightarrow\infty}\liminf_{n\rightarrow\infty}\frac{1}{n}\log\#K_n(\beta,l)\geq h_{\mu}(X)-\beta.
		$$
	\end{itemize}
	Since $V_z(\delta/2)$ is the flow box of $N_z(\delta/2)$ with time in $(-\delta/2,\delta/2)$, we can assume $K_n(\beta,l)\subset N_z(\delta/2)$. 
	For every $x\in K_n(\beta,l)$, since $\phi_{m_x}(x)\in V_z(\delta/2)$, there exists $\tau_x\in(-\frac{\delta}{2},\frac{\delta}{2})$ such that
	$$
	\phi_{m_x+\tau_x}(x)\in N_z(\delta/2).
	$$
	This allows us to define the local Poincar\'e return map from a neighborhood of $x$ in $N_z(\delta)$ to a neighborhood of $\phi_{m_x+\tau_x}(x)$ in $N_z(\delta)$. For simplicity of symbols, we still use ${\cal P}_{x,\phi_{m_x+\tau_x}(x)}$ to denote this local Poincar\'e return map. Recall that $d(x,\phi_{m_x}(x))<\frac{1}{1000l}$, hence shrinking $\delta$ if necessary, we have 
	$$
	d(x,\phi_{m_x+\tau_x}(x))<\frac{2}{1000l}.
	$$
	Moreover, for $l$ sufficiently large, there exist infinitely many $n$ satisfying
	$$
	\#K_n(\beta,l)>\exp(n(h_{\mu}(X)-2\beta)).
	$$
	
	Recall that $x\in K_n(\beta,l)\subset\Delta_C$ is $(C,\eta,1,E)$-$\psi^*_t$-contracting and $\phi_{m_x}(x)\in \Delta_C\subset \Lambda_C$ is $(C,\eta,1,F)$-$\psi^*_t$-expanding. Theorem \ref{Thm:Stable-Mfd} implies there exists $C'>0$ independent on $n$ such that
	\begin{itemize}
		\item[--] the diameter of ${\cal P}_{x,\phi_{m_x+\tau_x}(x)}(W^s_{loc}(\orb(x))\cap N_z(\delta))$ is smaller than $10C'e^\frac{-\eta n}{2}\delta$, and
		$$
		{\cal P}_{x,\phi_{m_x+\tau_x}(x)}(W^s_{loc}(\orb(x))\cap N_z(\delta))\subset W^s_{loc}(\orb(\phi_{m_x+\tau_x}(x)))\cap N_z(\delta);
		$$
		\item[--] the diameter of ${\cal P}^{-1}_{x,\phi_{m_x+\tau_x}(x)}(W^u_{loc}(\orb(\phi_{m_x+\tau_x}(x)))\cap N_z(\delta))$ is smaller than $10C'e^\frac{-\eta n}{2}\delta$, and
		$$
		{\cal P}^{-1}_{x,\phi_{m_x+\tau_x}(x)}(W^u_{loc}(\orb(\phi_{m_x+\tau_x}(x)))\cap N_z(\delta))\subset W^u_{loc}(\orb(x))\cap N_z(\delta).
		$$
	\end{itemize}
	Notice here the sizes of stable and unstable manifolds of $x$ and $\phi_{m_x+\tau_x}(x)$ do not need to time the flow speed, because they are both close to $z$ and the constant $\delta$ was chosen to be very small comparing with $|X(z)|$.
	
	Taking $n$ large enough, we have $10C'e^\frac{-\eta n}{2}\delta\ll\delta$. So for every point $w\in W^s_{loc}(\orb(x))\cap N_z(\delta)$ and for any disk $D^{cu}(w)\subset N_z(\delta)$ centered at $w$ tangent to the small cone field of the continuation of $F$-bundle, the connected component of ${\cal P}_{x,\phi_{m_x+\tau_x}(x)}(D^{cu}(w))\subset N_z(\delta)$ containing ${\cal P}_{x,\phi_{m_x+\tau_x}(x)}(w)$ is still tangent to a small cone field of the continuation of $F$-bundle. The same fact holds for points in $W^u_{loc}(\orb(\phi_{m_x+\tau_x}(x)))\cap N_z(\delta)$, the bundle $E$ and ${\cal P}^{-1}_{x,\phi_{m_x+\tau_x}(x)}$.
	
	This implies for any $x,y\in K_n(\beta,l)$, the connected component of ${\cal P}_{x,\phi_{m_x+\tau_x}(x)}(N_z(\delta))$ containing $\phi_{m_x+\tau_x}(x)$ is crossing the connected component of ${\cal P}^{-1}_{y,\phi_{m_y+\tau_y}(y)}(N_z(\delta))$ containing $y$.

	\begin{claim*}
		For $n$ large enough, the connected component ${\cal P}^{-1}_{x,\phi_{m_x+\tau_x}(x)}(N_z(\delta))$ containing $x$ does not contain any other point $y\in K_n(\beta,l)$.
	\end{claim*}
	
	\begin{proof}[Proof of the claim]
		
		Assume that there exists $y\neq x$ with $y\in K_n(\beta,l)\cap{\cal P}^{-1}_{x,\phi_{m_x+\tau_x}(x)}(N_z(\delta))$. We show that $d_n(x,y)$ is smaller than $1/l$ which leads to a contradiction.
		
		Since $y\in K_n(\beta,l)\cap{\cal P}^{-1}_{x,\phi_{m_x+\tau_x}(x)}(N_z(\delta))$, the distance between $x$ and $y$ in the $F$-bundle direction is smaller than $C''e^{\frac{-\eta m_x}{2}}\delta$ for some constant $C''$. On the other hand, for the points $\phi_{m_x+\tau_x}(x)$ and ${\cal P}_{x,\phi_{m_x+\tau_x}}(y)$ in $N_z(\delta)$, their distance in the $E$-bundle direction is also smaller than $C''e^{\frac{-\eta m_x}{2}}\delta$.
		When $n$ is large enough, we have that $C''e^{\frac{-\eta m_x}{2}}\delta\leq C''e^{\frac{-\eta n}{2}}\delta<1/1000l$.
		
		Recall that $d(x,\phi_{m_x}(x))<1/1000l$ and $d(y,\phi_{m_y}(y))<1/1000l$ with $\phi_{m_y+\tau_y}(y)={\cal P}_{x,\phi_{m_x+\tau_x}}(y)$. Moreover, we have $d(x,\phi_{m_x+\tau_x}(x))<\frac{2}{1000l}$ and $d(y,\phi_{m_y+\tau_y}(y))<\frac{2}{1000l}$. These estimations imply 
		$$
		d(x,y)<\frac{1}{100l} \qquad {\rm and } \qquad d(\phi_{m_x+\tau_x}(x),\phi_{m_y+\tau_y}(y))<\frac{1}{100l}.
		$$
		For any $0<i<m_x$, the distance between $\phi_i(x)$ and $\phi_i(y)$ is smaller than $d(x,y)$ in the $E$-bundle direction, and smaller than $d(\phi_{m_x+\tau_x}(x),\phi_{m_y+\tau_y}(y))$ in the $F$-bundle direction. This implies that
		$$
		d_n(x,y)<\frac{1}{10l}.
		$$
		This is a contradiction.
	\end{proof}
	
	This estimation shows that there are at least $L_n=\#K_n(\beta,l)$ different mutually disjoint connected components of ${\cal P}^{-1}_{x_i,\phi_{m_{x_i}+\tau_{x_i}}(x_i)}(N_z(\delta))$ containing $x_i\in K_n(\beta,l)$ where $i=0,1,\cdots,L_n-1$. Denote them by
	$$
	R_0, R_1, \cdots, R_{L_n-1}.
	$$
	Moreover, the image ${\cal P}_{x_i,\phi_{m_{x_i}+\tau_{x_i}}(x_i)}(R_i)$ is crossing each  $R_j$ for any $i,j=0,1,\cdots,L_n-1$.
	
	This shows the Poincar\'e return map
	$$
	{\cal P}|_{N_z(\delta)}:\bigcup_{i=0}^{L_n-1}R_i\longrightarrow \bigcup_{i=0}^{L_n-1}{\cal P}_{x_i,\phi_{m_{x_i}+\tau_{x_i}}(x_i)}(R_i)
	$$
	is a horseshoe with $L_n$-components.
	
	\begin{claim*}
		The maximal invariant set $\Gamma_n$ contained in $\bigcup_{i=0}^{L_n-1}R_i$ is a hyperbolic set with respect to ${\cal P}$.
		This implies the suspension set
		$$
		\tilde{\Gamma}_n=\bigcup_{t\in\RR}\phi_t(\Gamma_n)
		$$
		is also hyperbolic for $\phi_t$.
	\end{claim*}
	
	\begin{proof}[Proof of the claim]
		The hyperbolicity of $\Gamma_n$ is derived from the following facts.
		\begin{itemize}
			\item $x_0,x_1,\cdots,x_{L_n-1}, \phi_{m_{x_0}}(x_0),\cdots,\phi_{m_{x_{L_n-1}}}(x_{L_n-1})\in\Delta_C$. All these points are both $(C,\eta,1,E)$-$\psi^*_t$-contracting and $(C,\eta,1,F)$-$\psi^*_t$-expanding.
			\item For every $y\in\Gamma_n$, its orbit satisfies
			$$
			\bigcup_{t\in\RR}\phi_t(y)\subset\bigcup_{i=1}^{L_n}\bigcup_{t\in[0,m_{x_i}]}B(\phi_t(x_i),\delta\cdot|\phi_t(x_i)|).
			$$
		\end{itemize}
		Then by Lemma \ref{Lem:Continuity}, any point $y\in\Gamma_n$  admits a hyperbolic splitting with respect to ${\cal P}$ which is close to the splitting
		$$
		T_{x_i}N_z(\delta)=[T_{x_i}N_z(\delta)\cap(E(x_i)\oplus\langle X(x)\rangle)]\oplus [T_{x_i}N_z(\delta)\cap(F(x_i)\oplus\langle X(x)\rangle)].
		$$
		Since the return time is uniformly bounded on $\Gamma_n$, the suspension set $\tilde{\Gamma}_n$ is hyperbolic with respect to $\phi_t$.
	\end{proof}

	Finally, we only need to estimate the return time of Poincar\'e map ${\cal P}$ restricted on every $R_i$ for $i=0,1,\cdots,L_n-1$.
	Notice that restricted on each $R_i$ where $i=0,1,\cdots,L_n-1$, for $j=1,\cdots,m_{x_i}-1$ the Poincar\'e map
	$$
	{\cal P}_{x_i,\phi_{j}(x_i)}:R_i\rightarrow N_{\phi_{j}(x_i)}(\delta|X(\phi_{j}(x_i))|)
	$$
	is well defined.
	
	By Lemma \ref{Lem:Time-control}, for any $y\in R_i$, the return time $t(y)$ of $y$ with respect to ${\cal P}_{x_i,\phi_{m_{x_i}+\tau_{x_i}}(x_i)}$ is bounded by
	$$
	t(y)\leq(1+\theta\delta)(m_{x_i}+1)\leq(1+\theta\delta)(1+\beta)n.
	$$

	So when we consider the horseshoe of the flow $\phi_t$,
	$$
	\tilde{\Gamma}_n=\bigcup_{t\in\RR}\phi_t(\Gamma_n),
	$$
	its topological entropy
	$$
	h_{\topo}(X,\tilde{\Gamma}_n)\geq\frac{1}{(1+\theta\delta)(1+\beta)n}\cdot\log L_n,
	$$
	where
	$$
	L_n=\#K_n(\beta,l)\geq\exp(n(h_{\mu}(X)-2\beta)).
	$$
	
	Recall that $\delta\ll \varepsilon/2\theta(h_{\mu}(X)-\varepsilon)$.
	For $\varepsilon>0$, by taking $\beta$ small enough and $n$ large enough, we have a hyperbolic horseshoe $\Lambda_{\varepsilon}$ of $\phi_t$, such that
	$$
	h_{\topo}(X,\Lambda_{\varepsilon})>h_{\mu}(X)-\varepsilon.
	$$
\end{proof}







Now we are in a position to prove Proposition~\ref{Prop:lower-semi-continuity of entropy}.

\begin{proof}[Proof of Proposition~\ref{Prop:lower-semi-continuity of entropy}]
	Let $X\in\mathscr{X}^*(M)$ with $h_{\topo}(X)>0$. For any $\varepsilon>0$, by the variational principle, there is a non-trivial ergodic measure $\mu$ of $X$ such that $h_{\mu}(X)>h_{\topo}(X)-\frac{\varepsilon}{2}$. 
	Then by Proposition~\ref{prop:LSC of entropy}, there is a hyperbolic horseshoe $\Lambda_{\varepsilon}$ close to $\supp(\mu)$ such that $h_{\topo}(X,\Lambda_{\varepsilon})>h_{\mu}(X)-\frac{\varepsilon}{2}$. Therefore we have that $h_{\topo}(X,\Lambda_{\varepsilon})>h_{\topo}(X)-\varepsilon$.
\end{proof}

\section{Intermediate entropy property of suspension flows over SFT}\label{Section:suspension flow of SFT}
In this section, we prove the suspension flow of subshifts of finite type(SFT) has intermediate entropy property, which implies Proposition \ref{Pro:intermediate entropy of full shift flow}. We first recall some definitions about symbolic dynamics.

Fix an integer $k\geq 2$, the symbolic space with $k$ symbols is defined as
$$\Sigma_k=\prod_{n=-\infty}^{+\infty}\{0,1,\cdots,k-1\}.$$
Any point $x\in\Sigma_k$ can be written as $(x_i)_{i=-\infty}^{+\infty}$ where $x_i\in\{0,1,\cdots,k-1\}$ is the $i$-th position of $x$. The metric on $\Sigma_k$ is given by
$$d(x,y)=\sum_{i=-\infty}^{+\infty}\frac{|x_i-y_i|}{k^{|i|}}, \ \ \forall x,y\in\Sigma_k,$$
where $x=(x_i)_{i=-\infty}^{+\infty}$ and $y=(y_i)_{i=-\infty}^{+\infty}.$


The shift map $\sigma:\Sigma_k\rightarrow\Sigma_k$ is defined as
$$(\sigma(x))_i=x_{i+1} \ \ \text{for any $x=(x_i)_{i=-\infty}^{+\infty}$.}$$
Obviously $\sigma$ is a homeomorphism and the symbolic dynamics $(\Sigma_k,\sigma)$ is called the \emph{full $k$-shift}.


Given a $k\times k$ $0$-$1$ matrix $A=[a_{ij}]_{i,j\in\{0,\cdots,k-1\}},$ where by $0$-$1$ matrix we mean $a_{ij}\in\{0,1\},$ define
$$\Sigma_A=\{(x_i)_{i=-\infty}^{+\infty}\in\Sigma_k: a_{x_i x_{i+1}}=1,\forall i\in\mathbb{Z}\}.$$
Note that $\Sigma_A$ is a $\sigma$-invariant and closed subset of $\Sigma_k.$ We call $(\Sigma_A,\sigma)$ the \emph{subshift of finite type} determined by the matrix $A,$ or an \emph{SFT} for simplicity.
In particular, the full shift $(\Sigma_k,\sigma)$ is an SFT determined by the matrix with all entries being $1.$

Recall that a $k\times k$ matrix $A=[a_{ij}]_{i,j\in\{0,\cdots,k-1\}}$ is called {\it non-negative} if $a_{ij}\geq 0$ for all $0\leq i,j\leq k-1$. A non-negative matrix $A$ is called {\it irreducible} if for any $0\leq i,j\leq k-1$, there is $m\in\mathbb{N}$ such that $a^{(m)}_{ij}>0$, where $a^{(m)}_{ij}$ is the $(i,j)$-entry of $A^m$. If the integer $m$ does not depend on $(i,j)$, then $A$ is called {\it aperiodic}.

Let $(\Sigma_A,\sigma)$ be an SFT and  $\varphi:\Sigma_A\rightarrow \mathbb{R}^+$ be a continuous function.
We define
$$\Sigma_A^\varphi=\{(x,t):x\in \Sigma_A, t\in [0,\varphi(x)]\}/(x,\varphi(x))\sim(\sigma(x),0).$$
The suspension flow over $(\Sigma_A,\sigma)$ is defined as $\sigma^\varphi_t:\Sigma_A^\varphi\rightarrow \Sigma_A^\varphi$ by $\sigma^\varphi_t(x,s)=(x,s+t)$.

\medskip

The following proposition states that the measure entropy of the suspension flow $(\Sigma_A^\varphi,\sigma_t^\varphi)$ has the intermediate property when $A$ is irreducible and the roof function $\varphi$ is continuous.

\begin{proposition}\label{Pro:intermediate entropy of suspension flow}
	Consider the subshift of finite type $(\Sigma_A,\sigma)$ generated by an irreducible $0$-$1$ matrix $A$. Assume $\varphi:\Sigma_A\rightarrow\mathbb{R}^+$ is a continuous roof function, and the suspension flow $(\Sigma_A^\varphi,\sigma_t^\varphi)$ has positive entropy $h_{\topo}(\sigma_t^\varphi,\Sigma_A^\varphi)>0$. Then for any constant $h\in(0,h_{\topo}(\sigma_t^\varphi,\Sigma_A^\varphi))$, there exists a $\sigma_t^\varphi$-ergodic measure $\tilde{\mu}$, satisfying that $h_{\tilde{\mu}}(\sigma_t^\varphi,\Sigma_A^\varphi)=h$.
\end{proposition}

\subsection{Preliminaries of symbolic dynamics}

Given any pair $(p,P),$ where $p=(p_0,\cdots,p_{k-1})$ is a probability vector and $P=[p_{ij}]_{i,j\in\{0,\cdots,k-1\}}$ is a stochastic matrix ($p_{ij}\geq0$ and $\sum\limits_{j=0}^{k-1}p_{ij}=1$ for $i=0,\cdots,k-1$) satisfying $\sum\limits_{i=0}^{k-1}p_ip_{ij}=p_j,$ there exists an invariant measure $\mu$ \emph{determined by $(p,P)$} in the following way: for any $m,n\in\mathbb Z,n\geq0,$
$$\mu([l_0,\cdots,l_{n}]_{m})=p_{l_0}p_{l_0l_1}\cdots p_{l_{n-1}l_n},$$
where $[l_0,\cdots,l_{n}]_{m}=\{(a_i)_{i=-\infty}^{+\infty}: a_{m+j}=l_j,j=0,\cdots,n\}$ which generate the product $\sigma$-algebra of the SFT.
Such $\mu$ is called a \emph{Markov measure}. 

For Markov measures there are two basic facts we will use in the following. For more details readers may refer to \cite{walters}.
\begin{theorem}\label{Thm:Markov measure}
Let $\mu$ be the Markov measure determined by $(p,P)$ where $p$ is a probability vector and $P$ is a stochastic matrix. Then
\begin{enumerate}
\item $\mu$ is ergodic if and only if the matrix $P$ is irreducible.
\item the entropy of $\mu$ is $-\sum_{i,j}p_ip_{ij}\log p_{ij}.$
\end{enumerate}
\end{theorem}


The Perron-Frobenius Theorem gives some important properties for non-negative matrices, see for example \cite{gantmacher,walters}.

\begin{theorem}[Perron-Frobenius Theorem]\label{Thm:Perron-Frobenius}
Let $A$ be a non-negative $k\times k$ matrix. Then the matrix $A$ admits a non-negative eigenvalue $\lambda$ satisfying the following properties:
\begin{enumerate}
\item\label{item:other eigenvalue} All the other eigenvalues of $A$ have absolute value no greater than $\lambda$.
\item\label{item:bound of lambda}  $\min\limits_{0\leq i\leq k-1}\sum\limits_{j=0}^{k-1} a_{ij}\leq \lambda\leq \max\limits_{0\leq i\leq k-1}\sum\limits_{j=0}^{k-1} a_{ij}$.
\item\label{item:eigenvector} Corresponding to $\lambda$, there is a non-negative left (row) eigenvector $u=(u_0,u_1,\cdots,u_{k-1})$ and a non-negative right (column) eigenvector $v=(v_0,v_1,\cdots,v_{k-1})^T.$
\item\label{item:irreducible} Furthermore if $A$ is irreducible, then
		\begin{itemize}
			\item[--] $\lambda$ is a simple eigenvalue of $A$ and all the other eigenvalues of $A$ have absolute value strictly smaller than $\lambda$,
			\item[--] the two eigenvectors $u$ and $v$ in item~\ref{item:eigenvector} are strictly positive (i.e. $u_i, v_i> 0$ for all $i$).
       \end{itemize}
\end{enumerate}	
\end{theorem}

Let $A$ be an irreducible matrix, by Theorem~\ref{Thm:Perron-Frobenius} there exists a simple eigenvalue $\lambda>0$
with  a left (row) eigenvector $u=(u_0,\cdots,u_{k-1})$ and a right (column) eigenvector $v=(v_0,\cdots,v_{k-1})^T$ that are strictly positive. We may assume $\sum_{i=0}^{k-1}u_iv_i=1.$
Let
\begin{equation}\label{Parry measure}
p_i=u_iv_i,\ \ \ \  p_{ij}=\dfrac{a_{ij}v_j}{\lambda v_i},\ \ \ 0\leq i,j\leq k-1.
\end{equation}
The Markov measure determined by (\ref{Parry measure}) is called \emph{Parry measure} which we denote as $\mu_A.$ The following theorem states that the Parry measure is the unique measure with maximal entropy.

\begin{theorem}[Theorem 8.10 in \cite{walters}]\label{Thm:Parry measaure}
If $(\Sigma_A,\sigma)$ is an SFT with $A$ being irreducible. Then the Parry measure $\mu_A$ is the unique measure with maximal entropy for $\sigma,$ that is $$h_{\mu_A}(\sigma,\Sigma_A)=h_{\topo}(\sigma,\Sigma_A)=\log\lambda.$$
\end{theorem}

\subsection{Intermediate entropy property for SFT}\label{subsec:intermediate entropy for SFT}

In this section we show that an SFT $(\Sigma_A,\sigma)$ has the intermediate entropy property when $A$ is irreducible.

\begin{lemma}\label{Lem:intermediate entropy for SFT}
Let $(\Sigma_A,\sigma)$ be an SFT determined by a $k\times k$ irreducible $0$-$1$ matrix $A.$  Then there is a continuous map $\mu(\cdot):[0,1]\rightarrow\mathcal M_{inv}(\sigma,\Sigma_A)$ satisfying the following:
	\begin{enumerate}
		\item\label{item:extremal point} $h_{\mu(0)}(\sigma,\Sigma_A)=0$ and $\mu(1)=\mu_A$ where $\mu_A$ is the unique measure with maximal entropy;
		\item\label{item:ergodicity} for any $t\in(0,1]$, the measure $\mu(t)$ is ergodic;
		\item\label{item:continuity of entropy} the entropy map induced by $\mu(\cdot)$
$$h:[0,1]\rightarrow [0,h_{\topo}(\sigma,\Sigma_A)]$$
		$$\text{ \ \ \ \ } t\mapsto h_{\mu(t)}(\sigma,\Sigma_A)$$
		is continuous.
	\end{enumerate}
\end{lemma}


\begin{proof}
Let $p=(p_0,\cdots,p_{k-1})$ and $P=[p_{ij}]_{i,j=0}^{k-1}$ be the probability vector and stochastic matrix in (\ref{Parry measure}). Since $A$ is irreducible, by Theorem \ref{Thm:Parry measaure} the Parry measure $\mu_A$ determined by $(p,P)$ is the measure with maximal entropy.

Note that  for any $i\in\{0,\cdots,k-1\}$ there is $l(i)\in \{0,\cdots,k-1\}$ such that $a_{il(i)}=1$. For any $t\in [0,1]$, define a  $k\times k$ matrix $P(t)=[p_{ij}(t)]_{i,j=0}^{k-1}$ as follows:
\begin{equation*}
p_{ij}(t) =
\begin{cases}
tp_{ij}, & \textrm{$j\neq l(i)$}\\
p_{il(i)}+\sum\limits_{\substack{0\leq r\leq k-1,\\r\neq l(i)}} (1-t)p_{ir}, & \textrm{$j=l(i)$}
\end{cases}.
\end{equation*}
Obviously $P(t)$ is a stochastic matrix.

Let $\lambda(t)$ be the non-negative eigenvalue of $P(t)$ given by Theorem~\ref{Thm:Perron-Frobenius}. Since $\sum\limits_{j=0}^{k-1}p_{ij}(t)=1$ for any $i=0,\cdots,k-1,$ by item \ref{item:bound of lambda} of Theorem \ref{Thm:Perron-Frobenius} we have $\lambda(t)=1.$ We  may choose $p(t)=(p_0(t),\cdots,p_k(t))$ the (non-negative) row eigenvector corresponding to $\lambda(t)$ to satisfy that (i) $\sum_{i=0}^{k-1}p_{i}(t) =1;$ (ii) $p(1)=(p_0,\cdots,p_{k-1});$ (iii) $p(t)=(p_0(t),\cdots,p_{k-1}(t))$ varies continuously with respect to $t.$

Let $\mu(t)$ be the Markov measure determined by $(p(t),P(t)).$	Obviously $\mu(t)$ is continuous with respect to $t\in[0,1].$ Note that for $t\in(0,1],$ we have $p_{ij}(t)>0$ if and only if $a_{ij}(t)>0,$ hence $P(t)$ is irreducible since $A$ is. By Theorem~\ref{Thm:Markov measure}, we have $\mu(t)\in\mathcal M_{inv}(\sigma,\Sigma_A)$ is ergodic for any $t\in(0,1].$ Thus item~\ref{item:ergodicity} is satisfied.

By Theorem \ref{Thm:Markov measure}, we have $h_{\mu(t)}(\sigma,\Sigma_A)=-\sum_{0\leq i,j\leq k-1}p_i(t) p_{ij}(t)\log p_{ij}(t).$ Hence $h_{\mu(t)}(\sigma,\Sigma_A)$ is continuous with respect to $t$ since the pair $(p(t),P(t))$ varies continuously with respect to $t,$  which is item~\ref{item:continuity of entropy}.  When $t=0$ the matrix $P(0)$ is a $0$-$1$ matrix and thus $h_{\mu(0)}(\sigma,\Sigma_A)=0.$ Moreover, we have $\mu(1)=\mu_A$ where $\mu_A$ is the Parry measure. This proves item~\ref{item:extremal point} .
\end{proof}




\subsection{Conjugation of suspension flows over SFT}
For an SFT $(\Sigma_A,\sigma),$ for $n\geq2$ we call $b=(i_0,\cdots,i_{n-1})$ an \emph{admissible word of length $n$} if $a_{i_ji_{j+1}}=1$ for $j=0,\cdots,n-1.$
Observe that $a_{ij}^{(n)}$ is the number of all admissible words of length $n$ beginning with $i$ and ending with $j,$ where $a_{ij}^{(n)}$ denotes the $(i,j)$-th entry of $A^n.$
Let $k_n=\sum_{0\leq i,j\leq k-1}a_{ij}^{(n)}.$ Denote $\Gamma_n$ the set of all admissible words of length $n,$ then $\# \Gamma_n=k_n.$

Given two topological dynamical systems $(X,T)$ and $(Y,S),$ $T$ is \emph{conjugate to} $S$ if there exists a homeomorphism $h: X\to Y$ such that $h\circ T=S\circ h$ where $h$ is called a \emph{conjugacy}.
The following Lemma states that an SFT  is always  conjugate to another which has more symbols.

\begin{lemma}\label{Lem:sequence of SFT}
Given an SFT $(\Sigma_A,\sigma).$  Then for each $n\geq2$ there exists a $k_n\times k_n$ $0$-$1$ matrix $A_n$ such that the SFT $(\Sigma_{A_n},\sigma)$  is conjugate to $(\Sigma_A,\sigma)$ by a conjugacy $g_n,$ where $k_n$ is defined as above. 	Moreover, the matrix $A_n$ is irreducible if and only if $A$ is.
\end{lemma}

\begin{proof}
Since $\#\Gamma_n=k_n,$ we write $\Gamma_n=\{b_0,b_1,\cdots,b_{k_n-1}\}$ where for each $i$ there is a unique admissible word $(i_0,\cdots,i_{n-1})$ corresponding to $b_i$.
Now we define a $k_n\times k_n$ $0$-$1$ matrix $A_n$ as follows: for any $0\leq i,j\leq k_n-1,$ assume the corresponding words $b_i$ and $b_j$ in $\Gamma_n$ are $b_i=(l_0,\cdots ,l_{n-1})$ and
$b_j=(m_0,\cdots, m_{n-1})$ respectively. Let
\begin{equation*}
a_{ij,n}=
\begin{cases}
1, &\text{if}\ l_{p+1}=m_p\ \text{for all}\ 0\leq p<n-1,\\
0, &\text{otherwise},
\end{cases}
\end{equation*}
where  $a_{ij,n}$ denotes the $(i,j)$-th entry of $A_n.$ Hence we obtain an SFT $(\Sigma_{A_n},\sigma).$

Now we explain how to construct the conjugacy $g_n$. Note that for each point $(\cdots,i_{-1},i_0,i_1,\cdots)$ in $\Sigma_A$, there corresponds a unique point $(\cdots,j_{-1},j_0,j_1,\cdots)$ in $\Sigma_{A_n}$ such that $$b_{j_l}=(i_{l-\lfloor\frac{n-1}{2}\rfloor},\cdots,i_{l},\cdots,i_{l+\lfloor\frac{n}{2}\rfloor})$$
where $\lfloor a\rfloor$ denotes the maximal integer not larger than $a$. Define $g_n:\Sigma_{A_n}\rightarrow \Sigma_A$ which maps $(\cdots,j_{-1},j_0,j_1,\cdots)$ to $(\cdots,i_{-1},i_0,i_1,\cdots)$. It is easy to see that $g_n$ is a homeomorphism such that $g_n\circ\sigma=\sigma\circ g_n$. Thus $(\Sigma_{A_n},\sigma)$ is conjugate to $(\Sigma_{A},\sigma)$ by $g_n$.

\vspace{2mm}
To show the equivalence of irreducibility between $A_n$ and $A$ we only need to notice the following basic fact:
{\it the matrix $A$ is irreducible if and only if for any $i,j\in\{0,1,\cdots,k-1\}$, there exists an admissible word $(a_0,a_1,\cdots,a_{t-1})$ of $(\Sigma_A,\sigma)$ with $t\geq 2$ satisfying $a_0=i$ and $a_{t-1}=j$. }

Now assume that $A$ is irreducible. Let $b_i=(l_0,\cdots ,l_{n-1})$ and $b_j=(m_0,\cdots, m_{n-1})$ be two words in $\Gamma_n$.
By the fact above, there exists an admissible word $(a_0,a_1,\cdots,a_{t-1})$ of $(\Sigma_A,\sigma)$ such that $a_0=l_{n-1}$ and $a_{t-1}=m_0,$ which implies that there exists an admissible word of length $t+n-1$ of $(\Sigma_{A_n},\sigma)$ beginning with $b_i$ and ending with $b_j.$ Hence $A_n$ is irreducible. Similarly we could obtain the irreducibility of $A_n$  from that of $A.$
\end{proof}


For $n\geq1$ and $t\in\mathbb Z,$ a {\it cylinder set} in the full $k$-shift $\Sigma_k$ is
$$[i_0,\cdots,i_{n-1}]_t=\{(x_i)_{i=-\infty}^{+\infty}: x_t=i_0, \cdots,x_{t+n-1}=i_{n-1}\}.$$
In particular, we denote by $[i]_0$ the set of elements of $\Sigma_k$ with $i$ being their $0$-th position.

The next lemma would play an important role in estimating the metric entropy of a suspension flow over an SFT in Proposition~\ref{Pro:intermediate entropy of suspension flow}.

\begin{lemma}\label{Lem:embedding of SFT}
	
Given a  $0$-$1$ irreducible matrix $A$ and a continuous roof function $\varphi:\Sigma_A\rightarrow \mathbb{R^+}.$ Let $\Sigma_{A_n}$ and $g_n$ be as in Lemma~\ref{Lem:sequence of SFT}. Then for any $\eta>0$ there exist $N\in\mathbb{N}$ and a strictly positive roof function $\varphi_{N}^{\prime}:\Sigma_{A_{N}}\rightarrow\mathbb{R^+}$ satisfying following properties:

\begin{enumerate}
\item $\varphi_{N}'$ is constant restricted on $[i]_0\cap\Sigma_{A_N}$ for any $i=0,\cdots,k_n-1,$
\item	for any $\mu\in\mathcal M_{inv}(\sigma,\Sigma_{A})$, one has
	$$\Big|\frac{h_{\mu_*}(\sigma,\Sigma_{A_N})}{\int\varphi_N' d\mu_*}-\frac{h_{\mu}(\sigma,\Sigma_{A})}{\int\varphi d\mu}\Big|<\eta,$$
where $\mu_*=({g^{-1}_{ _N}})_*(\mu).$
\end{enumerate}
	
Moreover, there exist a $0$-$1$ irreducible matrix $B$ and a constant roof function $\tau:\Sigma_B\rightarrow \mathbb{R}^+$ such that the suspension flow $(\Sigma_B^{\tau},\sigma_t^{\tau})$ is conjugate to $(\Sigma_{A_{N}}^{\varphi_{N}'},\sigma_t^{\varphi_N'})$.
	
	
\end{lemma}

\begin{proof}
Define $\varphi_n:\Sigma_{A_n}\to\mathbb R^+$ as $\varphi_n=\varphi\circ g_n.$ Then $\varphi_n$ is continuous since $\varphi$ is. We have the following fact.
	
\begin{claim*}\label{Claim:variation of varphi_n}
$\max\limits_{i\in\{0,\cdots,k_n-1\}}\sup_{x,y\in [i]_0\cap\Sigma_{A_n}} |\varphi_n(x)-\varphi_n(y)|\rightarrow 0$ as $n\rightarrow+\infty.$
\end{claim*}
\begin{proof}
 For any $n\geq2$ and each $i\in\{0,\cdots,k_n-1\}$ assume that the corresponding admissible word $b_{i}\in\Gamma_n$ is $[l_{0},\cdots,l_{n-1}].$  Hence
$$g_n([i]_0)=[l_{0},\cdots,l_{n-1}]_{-\lfloor\frac{n-1}{2}\rfloor}.$$
Thus $$\sup_{x,y\in [i]_0\cap\Sigma_{A_n}} |\varphi_n(x)-\varphi_n(y)|=\sup_{x,y\in[l_{0},\cdots,l_{n-1}]_{-\lfloor\frac{n-1}{2}\rfloor}\cap\Sigma_{A}} |\varphi(x)-\varphi(y)|\to0,\ n\to+\infty$$
by the continuity of $\varphi.$
\end{proof}

Now fix a small constant $\eta>0,$ the following claim gives $N$ and $\varphi_N'.$
\begin{claim*}\label{Claim:staircase function}
For the constant $\eta>0$, there exist a positive integer $N$, a function $\varphi_{N}':\Sigma_{A_{N}}\rightarrow \mathbb{R^+}$ and a constant $\tau>0$ satisfying the following properties.
\begin{enumerate}
\item $\varphi_{N}'\geq \varphi_{N};$
\item For each $i\in\{0,1,\cdots,k_N-1\}$, there exists $l_i\in\mathbb Z^+$ such that
    $\varphi_{N}'|_{[i]_0\cap\Sigma_{A_{N}}}=l_i\tau;$
\item For any $\mu\in\mathcal M_{inv}(\sigma,\Sigma_{A}),$ one has
$$\Big|\frac{h_{\mu_*}(\sigma,\Sigma_{A_{N}})}{\int\varphi_{N}' d\mu_*}-\frac{h_{\mu}(\sigma,\Sigma_A)}{\int\varphi_{N} d\mu}\Big|<\eta,$$ where $\mu_*=({g^{-1}_{ _N}})_*(\mu).$
\end{enumerate}
\end{claim*}
	
\begin{proof}

Since $A$ is irreducible, we have that $h_{\topo}(\sigma,\Sigma_A)>0$ by Theorem~\ref{Thm:Parry measaure}.
By conjugacy of $g_n,$ let $\min\varphi=\min\varphi_n=a>0.$ Take $\delta=\dfrac{\eta a^2}{h_{\topo}(\sigma,\Sigma_A)}.$
		
		By the claim above, we could take $N$ large such that for any $i\in\{0,1,\cdots,k_{N}-1\}$,
		$$\max_{x,y\in [i]_0\cap\Sigma_{A_N}} |\varphi_N(x)-\varphi_N(y)|<\frac{\delta}{4}.$$
Now we take a function $\varphi_{N}'$ satisfying the following two properties:
		\begin{itemize}
			\item[--]  for any $i\in\{0,1,\cdots,k_{N}-1\}$, $\varphi_{N}'$ is constant on $[i]_0\cap\Sigma_{A_{N}}$ such that
			$$\max_{x\in [i]_0\cap\Sigma_{A_{N}}}\varphi_{N}(x)\leq\varphi_{N}'|_{[i]_0\cap\Sigma_{A_{N}}}
			\leq\max_{x\in [i]_0\cap\Sigma_{A_{N}}}\varphi_{N}(x)+\frac{\delta}{4}.$$
			\item[--] the collections of numbers $\{\varphi_{N}'|_{[i]_0}\cap\Sigma_{A_{N}}\}_{i=0}^{k_{N}-1}$ are rationally related. To be precise, there exists a constant $\tau>0$ and  positive integers $\{l_i\}_{i=0}^{k_{N}-1}$ such that $\varphi_{N}'|_{{[i]_0}\cap\Sigma_{A_{N}}}=l_i\tau.$
		\end{itemize}
Obviously, $|\varphi_{N}'(x)-\varphi_{N}(x)|\leq\dfrac{\delta}{2}, \forall x\in\Sigma_{A_N}.$
		
Now for any $\mu\in\mathcal M_{inv}(\Sigma_{A},\sigma),$ notice that $\dfrac{h_{\mu}(\sigma,\Sigma_A)}{\int\varphi d\mu}=\dfrac{h_{\mu_*}(\sigma,\Sigma_{A_{N}})}{\int\varphi_{N} d\mu_*}$, then we have
\begin{align*}
		\Big|\frac{h_{\mu_*}(\sigma,\Sigma_{A_{N}})}{\int\varphi_{N}' d\mu_*}-\frac{h_{\mu}(\sigma,\Sigma_A)}{\int\varphi d\mu}\Big|
				&=\Big|\frac{(\int\varphi_{N}' d\mu_*-\int\varphi_{N} d\mu_*)h_{\mu_*}(\sigma,\Sigma_{A_{N}})}{\int\varphi_{N}' d\mu_*\int\varphi_{N} d\mu_*}\Big|
		\\
		&\leq \frac{\max_{x\in\Sigma_{A_{N}}}(\varphi_{N}'(x)-\varphi_{N}(x)) h_{\topo}(\sigma,\Sigma_{A_{N}})}{a^2}
		\\
&\leq\dfrac{\delta}{2}\cdot\dfrac{h_{\topo}(\sigma,\Sigma_A)}{a^2}<\eta.
		\end{align*}
	\end{proof}
	
To complete the proof of Lemma \ref{Lem:embedding of SFT} we now explain how to construct  the $0$-$1$ irreducible matrix $B.$  Recall that $\varphi_{N}'|_{[i]_0\cap\Sigma_{A_{N}}}=l_i\tau$ for any $i\in\{0,1,\cdots,k_N-1\}.$ Let $L=\sum_{i=0}^{k_N-1} l_i.$	
Denote $$\Gamma=\{0_0,\cdots,0_{l_0-1},\cdots i_0,\cdots,i_{l_i-1},\cdots,{(k_N-1)}_0,\cdots,{(k_N-1)}_{l_{k_{ _N}-1}-1}\}$$ and correspond them to $\{0,1,2,\cdots,L-1\}$.
	Let $B=[b_{ij}]$ be the $L\times L$ $0$-$1$ matrix such that
	\begin{itemize}
		\item[-]$b_{ij}=1,$ either if $i$ corresponds to $m_{\alpha}$ and $j$ corresponds to $m_{\alpha+1}$ for some $0\leq m\leq k_N-1$ and $0\leq \alpha\leq l_m-2$, or if $i$ corresponds to $m_{l_m-1}$ and $j$ corresponds to $n_0$ such that $a_{mn,N}=1$ where $a_{mn,N}$ is the $(m,n)$-th entry of $A_N;$
		\item[-] $b_{ij}=0,$ otherwise.
	\end{itemize}
Obviously,  $B$ is irreducible since $A_N$ is.
\vspace{2mm}
	
Now consider  the suspension flow $(\Sigma_B^\tau,\sigma_t^\tau)$ over $(\Sigma_B,\sigma)$  where $\tau:\Sigma_B\rightarrow \mathbb{R}^+$ is the constant function. Define a map $g:\Sigma_B^\tau\rightarrow \Sigma_{A_{N}}^{\varphi_{N}'}$ as follows: given $x\in\Sigma_B$ with the following form	
$$(\cdots,i_{-1}^{(0)},i_{-1}^{(1)},\cdots,i_{-1}^{(l_{i_{-1}}-1)},i_0^{(0)},i_0^{(1)},\cdots,i_0^{(l_{i_0}-1)},i_1^{(0)},i_1^{(1)},\cdots,i_1^{(l_{i_1}-1)},\cdots),$$
where $(\cdots,i_{-1},i_{0},i_{1},\cdots)$ is a point in $\Sigma_{A_N},$ assume the $0$-position of $x$ is $i_0^{(\alpha)}$, then $g$ maps $(x,t)\in \Sigma_B^\tau$ to $(x',t')\in \Sigma_{A_{N}}^{\varphi_{N}'}$
with $x'=(\cdots,i_{-1},i_{0},i_{1},\cdots)$ and $t'=\alpha\tau+t.$ Then $g$ is a homeomorphism and $g\circ \sigma_t^\tau=\sigma_t^{\varphi_N'}\circ g$. Thus $(\Sigma_B^\tau,\sigma_t^\tau)$ is conjugate to $(\Sigma_{A_{N}}^{\varphi_{N}'},\sigma_t^{\varphi_N'})$ by $g$. 
\end{proof}

\bigskip

\subsection{Proof of Proposition~\ref{Pro:intermediate entropy of suspension flow}}\label{Section:proof of proposition 1.5}
Now we are prepared to prove Proposition~\ref{Pro:intermediate entropy of suspension flow}. Firstly we briefly recall some basic facts about suspension flows. For more details one may refer to~\cite[Chapter 6]{parry-pollicott}.
Let $(\Sigma_A^\varphi,\sigma_t^\varphi)$ be a suspension flow over an SFT $(\Sigma_A,\sigma).$
There is a $1$-$1$ correspondence between $\mathcal M_{inv}(\sigma,\Sigma_A)$ and $\mathcal M_{inv}(\sigma_t^\varphi,\Sigma_A^\varphi):$
for any $\mu\in\mathcal M_{inv}(\sigma,\Sigma_A),$  there is a standard way lifting $\mu$ to an invariant measure $\tilde{\mu}$ of  $(\Sigma_A^\varphi,\sigma_t^\varphi),$ and every invariant measure of $(\Sigma_A^\varphi,\sigma_t^\varphi)$ can be obtained in this way from an invariant measure of $(\Sigma_A,\sigma).$
Moreover, $\tilde{\mu}$ is ergodic if and only if $\mu$ is.  In \cite{abramov},  Abramov revealed the relation between $h_{\tilde{\mu}}(\sigma_t^\varphi,\Sigma_A^\varphi)$ and $h_\mu(\sigma,\Sigma_A)$ through the following formula:
\begin{equation}\label{equ:entropy}
h_{\tilde{\mu}}(\sigma_t^\varphi,\Sigma_A^\varphi)=\dfrac{h_\mu(\sigma,\Sigma_A)}{\displaystyle\int \varphi d\mu}.
\end{equation}

In the following, for an invariant measure $\mu$ of an SFT we denote by $\tilde{\mu}$ the corresponding invariant measure of the suspension flow.

\begin{proof}[Proof of Proposition~\ref{Pro:intermediate entropy of suspension flow}]
Let $(\Sigma_A^\varphi,\sigma_t^\varphi)$ be a suspension flow over $(\Sigma_A,\sigma)$ with $h_{\topo}(\sigma_t^\varphi,\Sigma_A^\varphi)>0.$
Given $h\in (0,h_{\topo}(\sigma_t^\varphi,\Sigma_A^\varphi)),$ through the conjugacy $\{g_n\}_{n\geq0}$ in Lemma~\ref{Lem:sequence of SFT} we only need to find an ergodic measure $\tilde{\mu}$ of $(\Sigma_{A_n}^{\varphi_n},\sigma_t^{\varphi_n})$ satisfying $h_{\tilde{\mu}}(\sigma_t^{\varphi_n},\Sigma_{A_n}^{\varphi_n})=h$ for some $n\in\mathbb{N},$ where $\varphi_n=\varphi\circ g_n.$


Let $\eta=\min\Big\{\dfrac{h_{\topo}(\sigma_t^\varphi,\Sigma_A^\varphi)-h}{4},\dfrac{h}{4}\Big\}.$ By Lemma~\ref{Lem:embedding of SFT} and formula~(\ref{equ:entropy}), there exist $N\in\mathbb N$ and  $\varphi_{N}':\Sigma_{A_{N}}\rightarrow\mathbb{R}^+$ such that for any
$\mu\in\mathcal{M}_{inv}(\sigma,\Sigma_{A}),$ we have
\begin{equation}\label{estimation}
\Big|h_{\tilde\mu}(\sigma_t^{\varphi_N},\Sigma_{A_N}^{\varphi_N})-h_{\tilde\mu'}(\sigma_t^{\varphi_N},\Sigma_{A_N'}^{\varphi_N'})\Big|<\eta.
\end{equation}
where $\tilde{\mu}\in\mathcal{M}_{inv}(\sigma_t^{\varphi_N},\Sigma_{A_N}^{\varphi_N})$ and $\tilde{\mu}'\in\mathcal{M}_{inv}(\sigma_t^{\varphi_N'},\Sigma_{A_N}^{\varphi_N})$ are the lifting measures of $\mu$ of the two suspension flows respectively.

By the Variational Principle, we have
\begin{equation}\label{variational principle}
h_{\topo}(\sigma_t^{\varphi_N'},\Sigma_{A_N'}^{\varphi_N'})\geq h_{\topo}(\sigma_t^{\varphi_N},\Sigma_{A_N}^{\varphi_N})-\eta\geq  h+3\eta.
\end{equation}



Let $B$ and $\tau$ be the irreducible matrix and the corresponding constant function $\tau:\Sigma_B\rightarrow \mathbb{R}^+$ obtained from Lemma \ref{Lem:embedding of SFT} such that $(\Sigma_B^{\tau},\sigma_t^{\tau})$ is conjugate to $(\Sigma_{A_{N}}^{\varphi_{N}'},\sigma_t^{\varphi_N'})$ through a conjugacy $g.$
Note that by (\ref{equ:entropy}) we have $h_{\tilde\nu}(\sigma_t^\tau,\Sigma_B^\tau)=\tau h_{\nu}(\sigma,\Sigma_B),$ $\forall \nu\in\mathcal M_{inv}(\sigma,\Sigma_B).$ Thus we could lift $(\nu(s))_{s\in[0,1]}$ obtained from Lemma \ref{Lem:intermediate entropy for SFT} by applying to the SFT $(\Sigma_B,\sigma)$ to a continuous map $\tilde{\nu}(\cdot):[0,1]\rightarrow \mathcal M_{inv}(\sigma_t^\tau,\Sigma_B^\tau)$ such that
\begin{itemize}
\item[(i)] {\it ergodicity}: $\tilde{\nu}(s)$ is ergodic, $\forall s\in(0,1];$
\item[(ii)] {\it minimality and maximality}: $h_{\tilde{\nu}(0)}(\sigma_t^\tau,\Sigma_B^\tau)=0,\ \ h_{\tilde{\nu}(1)}(\sigma_t^{\tau},\Sigma_B^{\tau})=h_{\topo}(\sigma_t^{\tau},\Sigma_B^{\tau});$
\item[(iii)] {\it entropy continuity}: the map $$h:[0,1]\rightarrow [0,h_{\topo}(\sigma_t^\tau,\Sigma_B^\tau)]$$
  $$s\mapsto h_{\tilde{\nu}(s)}(\sigma_t^\tau,\Sigma_B^\tau)$$
  is continuous.
\end{itemize}

Let $\tilde{\mu}'(s)=g_*({\tilde{\nu}(s)}), s\in[0,1].$
Then the continuous map $\tilde{\mu}^\prime(\cdot):[0,1]\rightarrow \mathcal M_{inv}(\sigma_t^{\varphi_N'},\Sigma_{A_{N}}^{\varphi_{N}'})$ satisfies (i) (ii) (iii) as above
for the suspension flow $(\Sigma_{A_{N}}^{\varphi_{N}'},\sigma_t^{\varphi_N'}).$
\medskip

Now consider $\tilde{\mu}(s)\in\mathcal{M}_{inv}(\Sigma_{A_{N}}^{\varphi_{N}}, \sigma_t^{\varphi_N})$ corresponding to $\tilde{\mu}'(s),$ where by corresponding we mean that they are lifting measures of the same $\mu(s)\in\mathcal M_{inv}(\sigma,\Sigma_A).$
Obviously, the map $\tilde{\mu}(\cdot):[0,1]\rightarrow \mathcal M_{inv}(\sigma_t^{\varphi_N},\Sigma_{A_{N}}^{\varphi_{N}})$ satisfies the properties (i) {\it ergodicity} and  (iii) {\it entropy continuity} as above.
Moreover, by (\ref{estimation}) and (\ref{variational principle}) we have the following estimations:
\begin{equation}\label{equa:lower}
0\leq  h_{\tilde{\mu}(0)}(\sigma_t^{\varphi_N},\Sigma_{A_N}^{\varphi_N})\leq  h_{\tilde{\mu}'(0)}(\sigma_t^{\varphi_N'},\Sigma_{A_N}^{\varphi_N'})+\eta=\eta<h;
\end{equation}
\begin{equation}\label{equa:upper}
h_{\tilde{\mu}(1)}(\sigma_t^{\varphi_N},\Sigma_{A_N}^{\varphi_N})\geq  h_{\tilde{\mu}'(1)}(\sigma_t^{\varphi_N'},\Sigma_{A_N}^{\varphi_N'})-\eta= h_{\topo}(\sigma_t^{\varphi_N'},\Sigma_{A_N}^{\varphi_N'})-\eta\geq    h+2\eta.
\end{equation}
Combine (\ref{equa:lower}) and (\ref{equa:upper}) with the property of {\it entropy continuity} of $\tilde{\mu}(\cdot),$ we have that there exists $s_0\in(0,1)$ such that
$$h_{\tilde{\mu}(s_0)}(\sigma_t^{\varphi_N},\Sigma_{A_N}^{\varphi_N})= h.$$
Moreover, we have $\tilde{\mu}(s_0)\in\mathcal{M}_{inv}(\sigma_t^{\varphi_N},\Sigma_{A_N}^{\varphi_N})$ is ergodic by the {\it ergodicity} of $\tilde{\mu}(\cdot).$

Let $\tilde{\mu}=\tilde{\mu}(s_0)$ and we complete the proof of Proposition~\ref{Pro:intermediate entropy of suspension flow}.
\end{proof}

\flushleft{\bf Ming Li} \\
School of Mathematical Sciences and LPMC, Nankai University, Tianjin, 300071, P. R. China\\
\textit{E-mail:} \texttt{limingmath@nankai.edu.cn}\\

\flushleft{\bf Yi Shi} \\
School of Mathematical Sciences, Peking University, Beijing, 100871, P. R. China\\
\textit{E-mail:} \texttt{shiyi@math.pku.edu.cn}\\

\flushleft{\bf Shirou Wang} \\
Department of Mathematical and Statistical Sciences, University of Alberta, Edmonton T6G2G1, Alberta, Canada\\
\textit{E-mail:} \texttt{shirou@ualberta.ca}\\

\flushleft{\bf Xiaodong Wang} \\
School of Mathematical Sciences, Shanghai Jiao Tong University, Shanghai, 200240, P. R. China\\
\textit{E-mail:} \texttt{xdwang1987@sjtu.edu.cn}\\

\end{document}